\documentclass[12pt]{amsart}    
\usepackage[fontsize=14]{scrextend}  
        
\newtheorem{theorem}{Theorem}[section] 
\newtheorem{proposition}[theorem]{Proposition}
\newtheorem{lemma}[theorem]{Lemma}
\newtheorem{corollary}[theorem]{Corollary}
\theoremstyle{definition}

\usepackage{fullpage}  

\newcommand{\inv}{^{\raisebox{.2ex}{$\scriptscriptstyle-1$}}}

\usepackage[all]{xy}  
\usepackage{enumerate,dsfont,txfonts,mathptmx,txfonts,newtxtext}   

\usepackage[utf8]{inputenc}
\usepackage[T1]{fontenc}
\usepackage{pdfrender}

\usepackage[dvipsnames]{xcolor}   
\usepackage{xparse}
\usepackage{xr-hyper}
\definecolor{brightmaroon}{rgb}{0.76, 0.13, 0.28}
\usepackage[linktocpage=true,colorlinks=true,hyperindex,citecolor=Royal Blue,linkcolor=Royal Blue]{hyperref}

\begin{document}
\pdfrender{StrokeColor=black,TextRenderingMode=2,LineWidth=.01pt} 
 
\title{Primitive ideals and Jacobson's structure spaces of noncommutative semigroups} 

\author{Amartya  Goswami}
 
\address{Department of Mathematics and Applied Mathematics\\ University of Johannesburg\\ P.O. Box 524, Auckland Park 2006\\ South Africa}

\address{National Institute for Theoretical and Computational Sciences (NITheCS)\\ South Africa} 
 
\email{agoswami@uj.ac.za}
  
\begin{abstract}
The purpose of this note is to introduce primitive ideals of noncommutative semigroups and study some topological aspects of the corresponding structure spaces.
\end{abstract}  

\makeatletter
\@namedef{subjclassname@2020}{%
\textup{2020} Mathematics Subject Classification}
\makeatother


 
\maketitle 

\section*{Introduction} 
Since the introduction of primitive  rings in \cite{J45}, primitive ideals have shown their immense importance in understanding structural aspects of rings and modules \cite{J56, R88}, Lie algebras \cite{KPP12}, enveloping algebras \cite{D96,J83}, PI-algebras \cite{J75}, quantum groups \cite{J95}, skew polynomial rings \cite{I79}, and others. In \cite{J451}, Jacobson has introduced a hull-kernel topology (also known as Jacobson topology) on the set of primitive ideals of a noncommutative ring, and has obtained representations of biregular rings. This Jacobson topology also turns out to play a key role in representation of finite-dimensional Lie algebras (see \cite{D96}).
  
Compare to the above algebraic structures, after magmas (also known as groupoids), semigroups are the most basic ones. A detailed study of algebraic theory of semigroups can be found in one of the earliest  textbooks \cite{CP61} and   \cite{CP67} (see also \cite{G01, H92, H95}), whereas specific study of prime, semiprime, and maximal ideals of semigroups are done in \cite{A53, A81, PK92, S69}. Furthermore, various notions of radicals of semigroups have been studied in \cite{A75, G69, S76}. Readers may consider \cite{AJ84} for a survey on ideal theory of semigroups.  

The next question is of imposing topologies on various types of ideals of semigroups. To this end, hull-kernel topology on maximal ideals of (commutative)  semigroups has been considered in \cite{A62}, whereas the same on minimal prime ideals has been done in \cite{K63}. Using the notion of $x$-ideals introduced in \cite{A62}, although in \cite{H66} a study of general notion of structure spaces for semigroups has been done, but having the assumption of commutativity restricts it to only certain types of ideals of semigroups, and hence did not have a scope for primitive ideals.  

To best of author's knowledge, primitive ideals of semigroups has never been considered. The aim of this paper is to introduce primitive ideals of (noncommutative) semigroups and endow Jacobson topology on primitive ideals to study some topological aspects of them. In order to have the notion of primitive ideals of semigroups, we furthermore need a notion of a module over a noncommutative semigroup, which in general has also not been studied much. We hope this notion of primitive ideals introduced here will in future shade some light on the structural aspects of noncommutative semigroups.
 
\section{Primitive ideals} 

A \emph{semigroup} is a tuple $(S, \cdot)$ such that the binary operation $\cdot$ on the set $S$ is associative. For all $a, b\in S$, we shall write $ab$ to mean $a\cdot b$.  Throughout this work, all semigroups are assumed to be noncommutative. If a semigroup $S$ has an identity, we denote it by $1$ satisfying the property: $s1=s=1s$ for all $s\in S.$ If $A$ and $B$ are subsets of $S$, then by the \emph{set product} $AB$ of $A$ and $B$ we shall mean $AB=\{ab\mid a\in A, b\in B\}.$ If $A=\{a\}$ we write $AB$ as $aB$, and similarly for $B=\{b\}.$ Thus
$$AB=\cup\{ Ab\mid b\in B\}=\cup \{aB\mid a\in A\}.$$

A \emph{left} (\emph{right}) \emph{ideal} of a semigroup $S$ is a nonempty subset $\mathfrak{a}$ of $S$ such that $S\!\mathfrak{a}\subseteq \mathfrak{a}$ ($\mathfrak{a} \,S\subseteq \mathfrak{a}$). A \emph{two-sided ideal} or simply an \emph{ideal} is a subset which is both a left and a right ideal of $S$. In this work
the word ``ideal'' without modifiers will always mean two-sided ideal. If $X$ is a nonempty subset of a semigroup $S$, then the ideal $\langle X\rangle$ \emph{generated by} $X$ is the intersection of all ideals containing $X$. Therefore, $$\langle X\rangle =X\cup XS\cup SX\cup XSX.$$ We say an ideal $\mathfrak{a}$ is of \emph{finite character} if the generating set $X$ of $\mathfrak{a}$ is equal to the set-theoretic union of all the ideals generated by finite subsets of $X$. We assume all our ideals are of finite character. To define primitive ideals of a semi group $S,$ we require the notion of a module over $S$, which we introduce now.

A (\emph{left}) \emph{$S$-module} is an abelian group $(M,+,0)$ endowed with a map  $S\times M\to M$ (denoted by $(s,m)\mapsto sm$) satisfying the identities:
\begin{enumerate}[\upshape (i)]
\itemsep -.2em 
\item $s(m+m')=sm+sm';$
\item $(ss')m=s(s'm);$
\item $s0=0,$
\end{enumerate}
for all $s,s'\in S$  and for all $m, m'\in M$.
Henceforth the term ``$S$-module'' without
modifier will always mean left $S$-module. If $M$, $M'$ are $S$-modules, then an \emph{$S$-module homomorphism} from $M$ into $M'$ is a group homomorphism $f\colon M\to M'$ such that $f(sm)=sf(m)$ for all $s\in S$ and for all $m\in M.$ 
A subset $N$ of $M$ is called an $S$-\emph{submodule} of the module $M$ if
\begin{enumerate}[\upshape (i)]
\itemsep -.2em 
\item $(N,+)$ is a subgroup of $(M,+);$
\item for all $s\in S$ and for all $n\in N$, $sn\in N.$
\end{enumerate}
If $\mathfrak{a}$ is an ideal of $S$, then the additive subgroup $\mathfrak{a}M$ of $M$ generated by
the elements of the form $\{am \mid a \in \mathfrak{a},m \in M\}$ is an $S$-submodule. 
An $S$-module $M$ is called \emph{simple} (or \emph{irreducible}) if
\begin{enumerate}[\upshape (i)] 
\itemsep -.2em  
\item $S\!M=\left\{\sum s_im_i \mid s_i\in S, m_i\in M\right\}\neq 0.$
\item There is no proper $S$-submodule of $M$ other than $0$.
\end{enumerate}
A (\emph{left})  \emph{annihilator} of an $S$-module $M$ is $\mathrm{Ann}_S(M)=\{ s\in S\mid sm=0\;\;\text{for all}\;\; m\in M\}.$ When $M=\{m\},$ we write $ \mathrm{Ann}_S(\{m\})$ as $ \mathrm{Ann}_S(m)$.  

\begin{lemma}
An annihilator $\mathrm{Ann}_S(M)$ is an ideal of $S$.
\end{lemma}

\begin{proof}
For all $s\in S$ and for all $x\in  \mathrm{Ann}_S(M)$ we have $(sx)m=s(xm)=s0=0.$ 
\end{proof} 

A nonempty proper ideal $\mathfrak{p}$ is said to be \emph{primitive} if $\mathfrak{p}=\mathrm{Ann}_S(M)$ for some simple $S$-module $M$. We denote the set of primitive ideals of a semigroup $S$ by $\mathrm{Prim}(S)$. 
A nonempty proper ideal $\mathfrak{q}$ of a semigroup $S$ is said to be  \emph{prime} if for any two ideals $\mathfrak{a}$, $\mathfrak{b}$ of $S$ and 
$\mathfrak{a}\mathfrak{b}\subseteq \mathfrak{q}$ implies $\mathfrak{a}\subseteq \mathfrak{q}$ or $\mathfrak{b}\subseteq \mathfrak{q}$.
 
As it has been remarked in \cite{BM58}, it does not matter whether the product $\mathfrak{a}\mathfrak{b}$ of ideals $\mathfrak{a}$ and $\mathfrak{b}$ is defined to be the set of all finite sums $\sum i_{\alpha} j_{\alpha}$ (where $i_{\alpha}\in \mathfrak{a}$, $j_{\alpha}\in \mathfrak{b}$), or the smallest ideal of the semigroup $S$ containing all products $i_{\alpha} j_{\alpha}$, or merely the set of all these products. For rings, in \cite{B56}, the second
of these definitions has been used and in \cite{A54} the third. The proof of the following result is easy to verify.

\begin{lemma}\label{dint}  If $\mathfrak{a}$ and $\mathfrak{b}$ are any two ideals of a semigroup, then $\mathfrak{a}\mathfrak{b}\subseteq \mathfrak{a}\cap \mathfrak{b}.$
\end{lemma}

The following proposition gives an alternative  formulation of prime ideals of semigroups. For a proof, see 
\cite[Lemma 2.2]{PK92}.

\begin{proposition}\label{alpri}
Suppose $S$ is a semigroup. Then the following conditions are equivalent:
\begin{enumerate}[\upshape (i)]
\itemsep -.2em 
\item $\mathfrak{q}$ is a prime ideal of $S$.
\item $aSb \subseteq  \mathfrak{q}$ implies $a\in \mathfrak{q}$ or $b\in\mathfrak{q}$\; for all $a, b \in S.$  
\end{enumerate}  
\end{proposition}

Primitive ideals and prime ideals of a semigroup are related as follows.

\begin{proposition}\label{prtpr}
Every primitive ideal of a semigroup is a prime ideal.
\end{proposition}

\begin{proof}
Suppose $\mathfrak{p}$ is a primitive ideal and $\mathfrak{p}=\mathrm{Ann}_S(M)$ for some simple $S$-module $M$. Let $a, b \notin\mathrm{Ann}_S(M)$. Then $am\neq 0$ and $bm'\neq 0$ for some $m, m'\in M.$ Since $M$ is simple, there exists an $s\in S$ such that $s(bm')=m$. Then
$$(asb)m'=a(s(bm'))=am\neq 0,$$
and hence $asb \notin \mathrm{Ann}_S(M)$. Therefore, $\mathrm{Ann}_S(M)$ is a prime ideal by Lemma \ref{alpri}.
\end{proof}

In the next section we talk about Jacobson topology on the set of primitive ideals of a semigroup and discuss about some of the topological properties of the corresponding structure spaces.

\section{Jacobson topology}

We shall introduce Jacobson
topology in $\mathrm{Prim}(S)$ by defining a closure operator for the subsets of $\mathrm{Prim}(S)$. Once we have a closure operator, closed sets are defined as sets which are invariant under this closure operator. Suppose $X$ is a subset of $\mathrm{Prim}(S)$. Set $\mathcal{D}_X=\bigcap_{\mathfrak{q}\in X}\mathfrak{q}.$ We define the closure of the set $X$ as 
\begin{equation}\label{clop}
\mathcal{C}(X)=\left\{ \mathfrak{p}\in \mathrm{Prim}(S) \mid \mathfrak{p}\supseteq \mathcal{D}_X \right\}.
\end{equation}

If $X=\{x\}$, we will write $\mathcal{C}(\{x\})$ as $\mathcal{C}(x)$. We wish to verify that the closure operation defined  in (\ref{clop}) satisfies Kuratowski's closure conditions and that is done in the following
 
\begin{proposition}\label{ztp}
The sets $\{\mathcal{C}(X)\}_{X\subseteq  \mathrm{Prim}(S)}$ satisfy the following conditions:
\begin{enumerate}[\upshape (i)]  
\itemsep -.2em 
\item\label{clee} $\mathcal{C}(\emptyset)=\emptyset$,
\item\label{clxx} $\mathcal{C}(X)\supseteq X$,
\item\label{clclx} $\mathcal{C}(\mathcal{C}(X))=\mathcal{C}(X),$
\item\label{clxy} $ \mathcal{C}(X\cup Y)=\mathcal{C}(X)\cup \mathcal{C}(Y).$ 
\end{enumerate}
\end{proposition}

\begin{proof}
The proofs of (\ref{clee})-(\ref{clclx}) are straightforward, whereas for (\ref{clxy}), it is easy to see that $ \mathcal{C}(X\cup Y)\supseteq\mathcal{C}(X)\cup \mathcal{C}(Y).$ To obtain the the other inclusion, let $\mathfrak{p}\in \mathcal{C}(X\cup Y).$ Then
$$\mathfrak{p}\supseteq \mathcal{D}_{X\cup Y}=\mathcal{D}_X \cap \mathcal{D}_Y.$$
Since $\mathcal{D}_X$ and $\mathcal{D}_Y$ are ideals of $S$, by Lemma \ref{dint}, it follows that 
$$\mathcal{D}_X\mathcal{D}_Y\subseteq \mathcal{D}_X \cap \mathcal{D}_Y\subseteq \mathfrak{p}.$$
Since by Proposition \ref{prtpr}, $\mathfrak{p}$ is prime, either $\mathcal{D}_X\subseteq \mathfrak{p}$ or $\mathcal{D}_Y\subseteq \mathfrak{p}$ This means either $\mathfrak{p}\in \mathcal{C}(X)$ or $\mathfrak{p}\in \mathcal{C}(Y)$. Thus $ \mathcal{C}(X\cup Y)\subseteq\mathcal{C}(X)\cup \mathcal{C}(Y).$ 
\end{proof}

The set $\mathrm{Prim} (S)$ of primitive ideals of a semigroup $S$ topologized (the Jacobson topology) by the 
closure operator defined in (\ref{clop}) is called the \emph{structure space} of the semigroup $S$.
It is evident from (\ref{clop}) that if $\mathfrak{p}\neq \mathfrak{p}'$ for any two $\mathfrak{p}, \mathfrak{p}'\in \mathrm{Prim}(S)$, then $\mathcal{C}(\mathfrak{p})\neq \mathcal{C}(\mathfrak{p}').$ Thus

\begin{proposition}\label{t0a}
Every structure space $\mathrm{Prim}(S)$ is a $T_0$-space. 
\end{proposition} 

\begin{theorem}\label{csb}
If $S$ is a semigroup with identity then the structure space $\mathrm{Prim}(S)$ is compact. 
\end{theorem}

\begin{proof}
Suppose  $\{K_{ \lambda}\}_{\lambda \in \Lambda}$ is a family of closed sets of $\mathrm{Prim}(S)$  with $\bigcap_{\lambda\in \Lambda}K_{ \lambda}=\emptyset.$ Set 
$$\mathfrak{a}=\left\langle \bigcup_{\lambda \in \Lambda} \mathcal{D}_{K_{\lambda}}\right\rangle.$$
If $\mathfrak{a}\neq S,$ then we must have a maximal ideal $\mathfrak{m}$ of $S$ such that $\mathfrak{a}\subseteq \mathfrak{m}.$ Moreover, 
$$\mathcal{D}_{K_{\lambda}}\subseteq \mathfrak{a}\subseteq \mathfrak{m},$$
for all $\lambda \in \Lambda.$ Therefore $\mathfrak{m}\in \mathcal{C}(K_{\lambda})=K_{\lambda}$ for all $\lambda \in \Lambda$, a contradiction of our assumption. Hence $\mathfrak{a}=S,$ and the identity $1\in \mathfrak{a}.$ Since $\mathfrak{a}$ is of finite character we must have a finite subset $\{\lambda_{\scriptscriptstyle 1}, \ldots, \lambda_{\scriptscriptstyle n}\}$ of $\Lambda$ such that $1\in \bigcup_{i=1}^n \mathcal{D}_{K_{\lambda_i}}.$ This implies $\bigcap_{\lambda_i}K_{\lambda_i}=\emptyset,$ which establishes the finite intersection property.  
\end{proof} 

Recall that a nonempty closed subset $K$ of a topological space $X$ is \emph{irreducible} if $K\neq K_{\scriptscriptstyle 1}\cup K_{\scriptscriptstyle 2}$ for any two proper closed subsets  $K_{\scriptscriptstyle 1}, K_{\scriptscriptstyle 2}$ of $K$. A maximal irreducible subset of a topological space $X$ is called an
\emph{irreducible component} of $X.$ A point $x$ in a closed subset $K$ is called a \emph{generic point} of $K$ if $K = \mathcal{C}(x).$  

\begin{lemma}\label{lemprime}
The only irreducible closed subsets of a structure space $\mathrm{Prim}(S)$ are of the form: $\{\mathcal{C}(\mathfrak{p})\}_{\mathfrak{p}\in \mathrm{Prim}(S)}$.  
\end{lemma}
   
\begin{proof} 
Since $\{\mathfrak{p}\}$ is irreducible, so is  $\mathcal{C}(\mathfrak{p}).$  Suppose $\mathcal{C}(\{\mathfrak{a}\})$ is an irreducible closed subset of $\mathrm{Prim}(S)$ and $\mathfrak{a}\notin \mathrm{Prim}(S).$ This implies there exist ideals $\mathfrak{b}$ and $\mathfrak{c}$ of $S$ such that  $\mathfrak{b}\nsubseteq \mathfrak{a}$ and $\mathfrak{c}\nsubseteq \mathfrak{a}$, but $\mathfrak{b}\mathfrak{c}\subseteq \mathfrak{a}$. Then   
$$\mathcal{C}(\langle \mathfrak{a}, \mathfrak{b}\rangle)\cup \mathcal{C}(\langle \mathfrak{a},\mathfrak{c}\rangle)=\mathcal{C}(\langle \mathfrak{a},  \mathfrak{b}\mathfrak{c}\rangle)=\mathcal{C}(\mathfrak{a}).$$
But $\mathcal{C}(\langle \mathfrak{a},  \mathfrak{b}\rangle)\neq \mathcal{C}(\mathfrak{a})$ and $\mathcal{C}(\langle \mathfrak{a},  \mathfrak{c}\rangle)\neq \mathcal{C}(\mathfrak{a}),$ and hence $\mathcal{C}(\mathfrak{a})$ is not irreducible.
\end{proof}

\begin{proposition}
Every irreducible closed subset of $\mathrm{Prim}(S)$ has a unique generic point.
\end{proposition}

\begin{proof}
The existence of generic point follows from Lemma \ref{lemprime}, and the uniqueness of such a point follows from Proposition \ref{t0a}. 
\end{proof} 

The irreducible components of a structure space can be characterised in terms of minimal primitive ideals, and we have that in the following

\begin{proposition}\label{thmirre}
The irreducible components of a structure space $\mathrm{Prim}(S)$ are the closed sets $\mathcal{C}(\mathfrak{p})$, where $\mathfrak{p}$ is a minimal primitive ideal of $S$. 
\end{proposition}

\begin{proof}
If $\mathfrak{p}$ is a minimal primitive ideal, then by Lemma \ref{lemprime}, $\mathcal{C}(\mathfrak{p})$ is irreducible. If $\mathcal{C}(\mathfrak{p})$ is not a maximal irreducible subset of $\mathrm{Prim}(S)$, then there exists a maximal irreducible subset $\mathcal{C}(\mathfrak{p}')$ with $\mathfrak{p}'\in  \mathrm{Prim}(S)$ such that $\mathcal{C}(\mathfrak{p})\subsetneq \mathcal{C}(\mathfrak{p}')$. This implies that $\mathfrak{p}\in \mathcal{C}(\mathfrak{p}')$ and hence $\mathfrak{p}'\subsetneq \mathfrak{p}$, contradicting the minimality property of $\mathfrak{p}$.
\end{proof}

Recall that a semigroup is called \emph{Noetherian} if it satisfies the ascending
chain condition, whereas a topological space $X$ is called \emph{Noetherian} if the descending chain condition holds for closed subsets of $X.$ A relation between these two notions is shown in the following 

\begin{proposition}\label{fwn}
If a semigroup $S$ is Noetherian, then $\mathrm{Prim}(S)$ is a Noetherian space.  
\end{proposition}

\begin{proof}
It suffices to show that a collection of closed sets in $\mathrm{Prim}(S)$ satisfy descending chain condition. Let $\mathcal{C}(\mathfrak{a}_{\scriptscriptstyle 1})\supseteq \mathcal{C}(\mathfrak{a}_{\scriptscriptstyle 2})\supseteq \cdots$
be a descending chain of closed sets in $\mathrm{Prim}(S)$. Then, $\mathfrak{a}_{\scriptscriptstyle 1}\subseteq \mathfrak{a}_{\scriptscriptstyle 2}\subseteq \cdots$ is an ascending chain of ideals in $S.$ Since $S$ is Noetherian, the chain stabilizes at some $n \in \mathds{N}.$ Hence, $\mathcal{C}(\mathfrak{a}_{\scriptscriptstyle n}) = \mathcal{C}(\mathfrak{a}_{\scriptscriptstyle n+k})$ for any $k.$ Thus $\mathrm{Prim}(S)$ is Noetherian.
\end{proof}

\begin{corollary}
The set of minimal primitive ideals in a Noetherian semigroup is finite.
\end{corollary}

\begin{proof}
By Proposition \ref{fwn}, $\mathrm{Prim}(S)$ is Noetherian, thus $\mathrm{Prim}(S)$ has a finitely many irreducible
components. By Proposition \ref{thmirre}, every irreducible closed subset of $\mathrm{Prim}(S)$
is of form $\mathcal{C}(\mathfrak{p}),$ where $\mathfrak{p}$ is a minimal primitive ideal. Thus $\mathcal{C}(\mathfrak{p})$ is irreducible components if and only if $\mathfrak{p}$ is minimal primitive. Hence, $S$ has
only finitely many minimal primitive ideals. 
\end{proof}
 
\begin{proposition}\label{conmap}
Suppose $\phi\colon S\to T$ is a semigroup homomorphism and define the map $\phi_*\colon  \mathrm{Prim}(T)\to \mathrm{Prim}(S)$ by  $\phi_*(\mathfrak{p})=\phi\inv(\mathfrak{p})$, where $\mathfrak{p}\in\mathrm{Prim}(T).$ Then $\phi_*$ is a continuous map.
\end{proposition}
  
\begin{proof}
To show $\phi_*$ is continuous, we first show that $f\inv(\mathfrak{p})\in \mathrm{Prim}(S),$ whenever $\mathfrak{p}\in \mathrm{Prim}(T)$. Note that $\phi\inv(\mathfrak{p})$ is an ideal of $S$ and a union of $\mathrm{ker}\phi$-classes (see \cite[Proposition 3.4]{G01}. Suppose $\mathfrak{p}=\mathrm{Ann}_{T}(M)$ for some simple $T$-module. Then by the ``change of rings'' property of modules, $\phi\inv(\mathfrak{p})$ is the annihilator of the simple $T$-module $M$ obtained by defining $sm:=\phi(s)m$. Therefore $f\inv(\mathfrak{p})\in \mathrm{Prim}(S)$. Now consider a closed subset $\mathcal{C}(\mathfrak{a})$ of  $\mathrm{Prim}(S).$ Then for any $\mathfrak{q}\in \mathrm{Prim}(T),$ we have:
\begin{align*}
\mathfrak{q}\in \phi_*\inv (\mathcal{C}(\mathfrak{a}))\Leftrightarrow \phi\inv(\mathfrak{q})\in \mathcal{C}(\mathfrak{a})\Leftrightarrow \mathfrak{a}\subseteq \phi\inv(\mathfrak{q})\Leftrightarrow \mathfrak{q}\in\mathcal{C}(\langle \phi(\mathfrak{a})\rangle),
\end{align*}  
and this proves the desired continuity of $\phi_*$.
\end{proof}


\end{document}